\documentclass[a4paper, 10pt, reqno, dvipsnames]{amsart}
\usepackage[a4paper, margin=3.4cm, top=3cm, bottom=3cm]{geometry}
\usepackage[utf8]{inputenc}

\usepackage{microtype}

\usepackage[english]{babel}

\usepackage{adjustbox}

\usepackage{verbatim}
\usepackage{newcent}

\usepackage{xcolor}

\usepackage{longtable}
\usepackage{graphicx, amsmath, amsfonts, amssymb, amsthm, amscd, amsbsy, amstext, tikz,tikz-cd, mathrsfs, mathtools, enumerate, enumitem, tensor,xfrac}   
\usetikzlibrary{decorations.pathreplacing}
\usetikzlibrary{calc}

\usepackage{latexsym,ifthen,stmaryrd,fancyhdr,empheq,verbatim, epigraph}

\usepackage{capt-of}

\usepackage{dynkin-diagrams}

\usepackage[colorlinks=true, allcolors=black]{hyperref}

\usepackage{cleveref} 

 \usepackage{mathabx}
 \usepackage{bbm}
 \usepackage{pifont}
 \usepackage{manfnt}

\usepackage{tikz-cd}
\usepackage{mathtools}

\makeatletter
\@namedef{subjclassname@2020}{%
  \textup{2020} Mathematics Subject Classification}
\makeatother

\usepackage[OT2, T1]{fontenc} 

\usepackage[textsize=scriptsize]{todonotes}

\usepackage{scalerel, stackengine}

\usepackage{graphicx}

\usepackage{todonotes} 

\mathtoolsset{
  showmanualtags}    

\hypersetup{anchorcolor=black, linktocpage=false,
colorlinks=true,
citecolor=RawSienna,
linkcolor=Mahogany,
urlcolor=black,
breaklinks=true,
plainpages=true
}

\usetikzlibrary{positioning,shapes,shadows,arrows}
\usepackage[mathscr]{euscript} 
\allowdisplaybreaks

\DeclareRobustCommand{\SkipTocEntry}[5]{}

\usepackage{calligra}

\DeclareMathAlphabet{\mathcalligra}{T1}{calligra}{m}{n}
\DeclareFontShape{T1}{calligra}{m}{n}{<->s*[2.2]callig15}{}
\newcommand{\scriptr}{\mathcalligra{r}\,}

\usepackage[all]{xy}
\usepackage{scalerel}
\usepackage{graphicx}

\newtheorem{theorem}{Theorem}

\newtheorem{prop}[theorem]{Proposition}
\newtheorem{lemma}[theorem]{Lemma}
\newtheorem{coro}[theorem]{Corollary}

\newtheorem*{corollary*}{Corollary}
\newtheorem*{theorem*}{Theorem}
\newtheorem*{rmk*}{Remark}
\newtheorem*{proposition*}{Proposition}
\newtheorem*{conjecture*}{Conjecture}
\newtheorem*{deff*}{Definition}
\numberwithin{equation}{section}
\numberwithin{theorem}{section}
\newtheorem{example}[theorem]{Example}
\theoremstyle{remark}
\newtheorem{rmk}[theorem]{Remark}
\newtheorem{notation}[theorem]{Notation}
\newtheorem{construction}[theorem]{Construction}

\theoremstyle{definition}

\newtheorem*{defi*}{Definition}

\newtheorem{deff}[theorem]{Definition}

\newcommand{\cC}{{\mathcal C}}

\newcommand{\cE}{{\mathcal E}}

\newcommand{\cL}{{\mathcal L}}
\newcommand{\cM}{{\mathcal M}}
\newcommand{\cN}{{\mathcal N}}

\newcommand{\cP}{{\mathcal P}}

\newcommand{\cR}{{\mathcal R}}
\newcommand{\cS}{{\mathcal S}}

\newcommand{\cW}{{\mathcal W}}

\newcommand{{\bull}}{{\scriptscriptstyle{\bullet}}}

\newcommand{\Op}{\mathsf{Op}}
\newcommand{\Cat}{\mathsf{Cat}}
\newcommand{\alg}{\mathsf{Alg}}

\newcommand{\omg}{\ensuremath{\mathbf{\Omega}}}

\newcommand{\ssets}{\ensuremath{\mathbf{sSets}}}
\newcommand{\dsets}{\ensuremath{\mathbf{dSets}}}

\DeclareMathOperator{\colim}{\mathsf{colim}}

\newcommand{\id}{\text{id}}

\newcommand{{\op}}{{{\rm op}}}
\newcommand{{\coop}}{{{\rm coop}}}

\usepackage{pict2e}

\makeatletter
\newcommand{\adjunction}[4]{%
  #1\colon #2%
  \mathrel{\vcenter{%
    \offinterlineskip\m@th
    \ialign{%
      \hfil$##$\hfil\cr
      \longrightharpoonup\cr
      \noalign{\kern-.3ex}
      \smallbot\cr
      \longleftharpoondown\cr
    }%
  }}%
  #3 \noloc #4%
}
\newcommand{\longrightharpoonup}{\relbar\joinrel\rightharpoonup}
\newcommand{\longleftharpoondown}{\leftharpoondown\joinrel\relbar}
\newcommand\noloc{%
  \nobreak
  \mspace{6mu plus 1mu}
  {:}
  \nonscript\mkern-\thinmuskip
  \mathpunct{}
  \mspace{2mu}
}
\newcommand{\smallbot}{%
  \begingroup\setlength\unitlength{.15em}%
  \begin{picture}(1,1)
  \roundcap
  \polyline(0,0)(1,0)
  \polyline(0.5,0)(0.5,1)
  \end{picture}%
  \endgroup
}
\makeatother

\begin{document}

\title{The root functor} 

\author{Francesca Pratali}

\begin{abstract}  In this paper, we show that any $\infty$-operad is equivalent to the localization of a discrete $\Sigma$-free operad; this result extend Joyal's delocalization theorem for categories to the operadic setting.  Along the way, we pursue a systematic study of $\infty$-operadic localization in the dendroidal context and its compatibility with un/straightening equivalences, deducing another description of algebras over $\infty$-operads.

\end{abstract}

%

\keywords{Localization, $\infty$-operads, dendroidal sets, covariant model structure.}


\maketitle

\tableofcontents

\section*{Introduction and main results}
\noindent Localization of $\infty$-categories is a central construction in homotopy theory. Given
an $\infty$-category $\cC$ and a class of maps $\cW$, thought of as weak equivalences,
we can construct an $\infty$-category $\cC[\cW^{-1}]$ obtained by freely inverting
the arrows in $\cW$. Combining the results of Barwick--Kan \cite{BK:RCAMHTHT} and Mazel-Gee \cite{MG:URN}, the localization of the nerve of a relative category, i.e. a discrete category with a choice of weak equivalences, induces an equivalence of $\infty$-categories between that of $\infty$-categories and that of relative categories.
Joyal's \emph{delocalization theorem}, formulated in \cite[\S13.6]{J:NQC} and proven
by Stevenson \cite{Ste:CMSSL}, gives a homotopy inverse for this equivalence. It is an explicit construction which makes use of the model of quasicategories. Indeed, to any simplicial set $X$ one has associated its \emph{category of elements} $\Delta/X$ : its objects are the simplices of $X$, that is maps $([n],\alpha\colon \Delta^n \to X)$, and morphisms are commutative triangles. Joyal's delocalization theorem states the following

\begin{theorem*}[Joyal, Stevenson]
	There is a natural transformation of simplicial sets
	$$\ell_X\colon \cN(\Delta/X) \longrightarrow X, \quad ([n],\alpha)\mapsto\alpha(n),$$
	sending each simplex to its last vertex, which is a weak categorical equivalence
	upon localizing $\cN(\Delta/X)$ at the preimage of the identities.
\end{theorem*}

\noindent In the context of \emph{higher algebra}, that is the study of algebraic structures in $\infty$-categories,  of key interests becomes the localization for $\infty$-operads. These latter generalize $\infty$-categories as colored operads, or multicategories, generalize ordinary categories by
allowing operations with multiple inputs \cite{M:TGILS, BV:HIASTS}. In fact, the language of
$\infty$-operads precisely allows to lift this structure to the homotopy-coherent setting,
providing a convenient framework for organizing algebraic structures in modern homotopy
theory \cite{GW:EPOVITII, W:EPOVITI, BdBW:MCHS, KK:IOFEC, Lu:HA}.

\noindent Informally, an $\infty$-operad $\cP$ consists of a collection of objects and, for every
choice of objects $x_1,\dots,x_n,y$, a space, or homotopy type, of operations $\cP(x_1,\dots,x_n;y)$,
together with operadic partial composition laws
$$
\circ_{x_i}\colon 
\cP(x_1,\dots,x_k;y)\times \cP(z_1,\dots,z_m;x_i)
\longrightarrow 
\cP(x_1,\dots,x_{i-1},z_1,\dots,z_m,x_{i+1},\dots,x_k;y),
$$ which are well defined only 'up to homotopy' and
compatible with the symmetric group actions permuting the inputs. We say that $\cP$ is \emph{discrete} if it is $0$-truncated,
i.e.\ if all spaces of operations are sets. The dendroidal formalism
\cite{MW:DS, CM:DSMHO} models $\infty$-operads as certain presheaves on a category $\Omega$
of trees (\cite{MW:DS}), and is in direct analogy with simplicial sets for quasicategories. They are a model for $\infty$-operads as simplicial sets and quasicategories are for  $\infty$-categories (\cite{CM:DSMHO}). The composition $p\circ_{x_1}q$ of two multimorphisms of a quasioperad looks like a section of two natural transformations obtained as

\begin{figure}[h]
	\centering
	
	\adjustbox{scale=0.63,center}{%

\begin{tikzcd}
	&                                                &     &                         &                                                &                                           &  & z_1 &                                                                                        & z_2                                              &  & x_2           &  &                         &                                                                   &     \\
	x_1 \arrow[rd, no head] &                                                & x_2 & z_1 \arrow[rd, no head] &                                                & z_2                                       &  &     & \bullet \ q \ \arrow[rd, "x_1" description, no head] \arrow[lu, no head] \arrow[ru, no head] &                                                  &  &               &  & z_1 \arrow[rd, no head] & z_2                                                               & x_2 \\
	& \bullet \ p \  \arrow[ru, no head] \arrow[d, no head] &     &                         & \bullet \ q \arrow[ru, no head] \arrow[d, no head] & \quad\quad \arrow[rrr, dashed, bend left] &  &     & \quad\quad\quad\quad \arrow[lll]                                                       & \bullet \ p \arrow[rruu, no head] \arrow[d, no head] &  & {} \arrow[rr] &  & {}                      & \bullet \ p\circ_{x_1}q \arrow[u, no head] \arrow[ru, no head] \arrow[d, no head] &     \\
	& y                                              &     &                         & x_1                                            &                                           &  &     &                                                                                        & y                                                &  &               &  &                         & y                                                                 &    
\end{tikzcd}

	}
\end{figure}

\noindent The localization of an $\infty$-operad $\cP$ at a class of morphisms $\cW$ is defined by a universal property analogous to that for categories. In particular, the $\infty$-category of $\cP[\cW]^{-1}$-algebras is the full subcategory of $\cP$-algebras
sending the morphisms in $\cW$ to equivalences, the \emph{locally constant algebras}.
This perspective appears naturally in the study of the little $n$-disks operad
$\mathbb{E}_n$: by \cite[Theorem~5.4.5.9]{Lu:HA}, $\mathbb{E}_n$ is equivalent to
the localization of the discrete operad of disks in $\mathbb{R}^n$ via the canonical
map $\cN\mathsf{Disk}_n^\otimes \longrightarrow \mathbb{E}_n$, a foundational result
for factorization homology.
Analogous statements hold for disk operads on stratified
or framed manifolds \cite{AFT:FHSS, KSW:ACFA}, playing a key role in identifying certain colimits and establishing the existence of operadic Kan extensions.

\noindent Our guiding question is the following.

\begin{center}
	\emph{\textbf{Question}. Can we generalize Joyal's delocalization theorem to $\infty$-operads and prove
		that every $\infty$-operad is equivalent to a localization of a discrete one?}
\end{center}

\noindent In this paper, we give a positive answer and expand on the study of
dendroidal localization and its relation with un/straightening theorems, in particular with dendroidal left fibrations. A more precise overview of
our results is given in the next section.

\medskip

\subsection*{Main results}

\noindent The category $\dsets$ of \emph{dendroidal sets} (the necessary preliminaries are in Chapter \ref{section1}) is that of presheaves on the \emph{dendroidal category} $\omg$ : its objects are trees, and maps can be defined via a fully faithful inclusion $\omg\subset\Op$. In particular, there is an inclusion $\Delta\subset \omg$ and a nerve functor $$\cN_d\colon \Op \to \dsets $$ for discrete operads.

\noindent Given a dendroidal set $X$, we construct its \emph{operad of elements} $\omg/X^\otimes$ (Definition \ref{operadofdendrices}), fitting into a commutative diagram of functors
\vspace{-0.2cm}
\[\begin{tikzcd}
	\dsets \arrow[r, "(\omg/-)^\otimes"]                   & \Op \arrow[r, "\cN_d"]                & \dsets                 \\
	\ssets \arrow[u, hook] \arrow[r, "\Delta/"] & \Cat \arrow[u, hook] \arrow[r, "\cN"] & \ssets  . \arrow[u, hook]
\end{tikzcd}
\]

\noindent An object $(T,\alpha)$ of $(\omg/X)^\otimes$ is given by a tree $T$ and $\alpha\colon T \to X$ a representable in $X$ ; it can be thought of as a scheme of composable multimorphisms in $X$ and higher coherences between them. Evaluating an object $(T,\alpha)$ at the root of $T$ yields an assignemt $$\mathsf{Ob}(\omg/X)^\otimes\to X_0 \ (\ast),$$ which we prove to extend to a natural transformation $$ \scriptr_X\colon \cN_d(\omg/X)^\otimes\longrightarrow X,$$ cocontinuously in $X$. We call this map the \emph{root functor} (\Cref{rootfnc}).
In fact, $(\omg/X)^\otimes$ is the maximal suboperad spanned by the representables in the cocartesian operad $(\dsets_{/X})^\sqcup$ $(\omg/X)^\sqcup$ for which the assignment $(\ast)$ extends as just described.  Our main result is the following.

\begin{theorem*}[\ref{main}]\label{theorem:root}
	If $X$ is normal (Definition \ref{def:normal}), the root functor $\scriptr_X$ induces an operadic weak equivalence
	$$\overline{\scriptr_X}\colon \cN_d(\omg/X)^\otimes[\cR^{-1}]\xlongrightarrow{\sim} X$$ where the morphisms in $\cR$ are the preimages of the identities.
	
	\noindent If $X$ is a simplicial set, one recovers Joyal's delocalization theorem, in the sense that ${(\omg/X)^\otimes\simeq \Delta/X}$ and ${\scriptr_X\simeq \ell_X.}$
\end{theorem*}	

\noindent The core of the proof of \Cref{main} lies in the construction of the operad of
elements $(\Omega/X)^\otimes$; once this is in place, the strategy is inspired by the simplicial case
\cite{Ste:CMSSL}, though the dendroidal setting requires several arguments to be carefully refined. Finally, it is natural to ask what happens when $X$ is a symmetric monoidal $\infty$-category, also in light of \cite{A:MRCMMIC} ; this is addressed in Remark~\ref{rmk:smcat}.

\noindent From the delocalization theorem we deduce that in the stable model structure on
dendroidal sets \cite{BN:DSMCS}, any $\infty$-operad $X$ is weakly equivalent to
the infinite loop space of $(\Omega/X)^\otimes$ (Corollary~\ref{westable}).
Specializing to $\infty$-categories recovers the well-known fact that every
$\infty$-category has the weak homotopy type of a nerve, already used in \cite{Wa:AKTSAGT}.

\noindent In the last \Cref{application}, we focus on operadic algebras and operadic left fibrations ; these two, organized in appropriate $\infty$-categories, are well known to be equivalent via the un/straightening equivalences (\cite{He:AOIO}, \cite{BdBM:DSGSSBPQT}, \cite{R:MGCIC}, \cite{K:MEGCDSO}, \cite{P:SUEIO}). More precisely, we consider a localization of dendroidal sets $X\to X[S^{-1}]$ and study the model category structure on $\dsets/X$ obtained as a left Bousfield localization of the covariant model structure (\Cref{construction}). Its fibrant objects are those dendroidal left fibrations $E\to X$ for which a morphism $f\colon x\to y$ in $S$ induces a weak homotopy equivalence of Kan complexes $f_!\colon E_x \to E_y$ between the fibres.  We denote it by $\cL_S\left(\dsets/X\right)$. The main result of Section \ref{application} is the following.

\begin{proposition*}[{\ref{localg}}]
	For every localization map of dendroidal sets $\lambda\colon X \to X[S^{-1}]$, the pullback along $\lambda$ induces a Quillen equivalence of model categories $$ \adjunction{\lambda_!}{\cL_S\left(\dsets/X\right)}{\dsets/X[S^{-1}]}{\lambda^*}.$$ 
\end{proposition*}

\noindent Specifying the result to the case of the nerve of the operad of elements allows to deduce local un/straightening equivalences (of semimodel categories), as we write in Corollary \ref{corofip}.

\addtocontents{toc}{\SkipTocEntry}

\subsection*{Acknowledgements}I thank Gijs Heuts for suggesting this problem and for many valuable discussions,
and Miguel Barata, Victor Carmona, Denis-Charles Cisinski and Ieke Moerdijk for
several fruitful exchanges. I am grateful to my PhD advisor Eric Hoffbeck for his
constant support and for carefully reading my drafts. This project was realized
during my PhD at Paris~13 and a visit to Utrecht, funded by the EOLE grant of
the French-Dutch Network (RFN) and the Marie Skłodowska-Curie grant No~945332 (EU Horizon 2020).
\section{Preliminaries}\label{section1}

\noindent In this section, we recall the necessary preliminary notions and set up notation. Complete proofs and definitions can be found in \cite{HeMo:SDHT}.

\subsection{Discrete operads}

\noindent From now on, we will call \emph{operad} a discrete, strict operad $P$. We write $\Op$ for the category of operads and morphisms between them. Observe that the set $\mathsf{Ob}(P)$ has the structure of a poset, where $c\leq d$ if and only if $P$ has an operation with target $d$ and set of inputs containing $c$. Observe that a morphism of operads $f\colon P \to Q$ induces a morphism of posets $f\colon \mathsf{Ob}(P)\to \mathsf{Ob}(Q)$.

\noindent There is a fully faithful functor $j_!\colon \Cat \to \Op$, where $j_!$ acts by extension by zero, i.e. n the sense that $j_!C(c_1,\dots,c_n;d)=\emptyset$ if $n\neq 1$. \noindent Its right adjoint $j^*\colon \Op \to \Cat$ is called the \emph{underlying category} functor.

\begin{deff}
	An operad $P$ is \emph{$\Sigma$-free} if, for any natural number $n$ and any choice of objects $c_1,\dots, c_n, c$, the symmetric group on $n$-elements $\Sigma_n$ acts freely on the set $	\underset{{\sigma \in \Sigma_n}}{\bigcup} P(c_{\sigma(1)},\dots,c_{\sigma(n)};c).$
\end{deff}

\noindent Given a symmetric monoidal category $(\mathbb{V},\otimes)$, one forms the operad
$\mathbb{V}^\otimes$ with operations $\mathbb{V}^\otimes(x_1,\dots,x_n;y)=
\mathbb{V}(x_1\otimes\cdots\otimes x_n; y)$. A $P$-algebra in $\mathbb{V}$ is
then a morphism of operads $P\to \mathbb{V}^\otimes$.

\noindent Denote by $\alg_P(\mathbb{V})$ the category of $P$-algebras in $\mathbb{V}$ and morphisms between them. Observe that for a category $C$, we have a natural identification $\alg_{j_!C}(\mathbb{V})=\mathsf{Fun}(C,\mathbb{V})$.

\subsection{The dendroidal category}\label{cat:omg}  Let $\Delta$ be the category of finite linear orders ${[n]=\{0<1< \dots < n\}}$, $n\geq 0$, and morphisms the maps of posets. Simplicial sets $\ssets$ are the category of presheaves over $\Delta$. There is a fully faithful functor $\Delta\to \Cat$, which by restricted Yoneda embedding induces the nerve functor $\cN\colon \Cat \to \ssets$.

\noindent The dendroidal category $\omg$ has objects non planar finite rooted trees.

\adjustbox{scale=0.6,center}{
	\begin{tikzcd}
		& && {}&&&&&&&\\
		&& {} \arrow[rd, no head] & \bullet \arrow[u, no head] &{}&&{} \arrow[rd, no head]&{}&{}&& {} \arrow[dd, no head] \\
		T = & {} & {} \arrow[rd, no head] & \bullet \arrow[ru, no head] \arrow[d, no head] \arrow[u, no head]& \bullet & {} & C_3 \ =& \bullet \arrow[ru, no head] \arrow[u, no head] \arrow[d, no head] & & \eta \ = &\\
		& && \bullet \arrow[d, "r_T", no head] \arrow[ru, no head] \arrow[llu, no head] \arrow[rru, no head] &&&                        & {}&&&{} \\
		&   &                        & {}                                                                                              &         &        &                        &                                                                   &     &          &                       
\end{tikzcd}}
\vspace*{-.3cm}

\captionof{figure}{Some typical trees in $\omg$.}

\noindent Any such tree $T$ yields an operad $\Omega(T)$, whose set of objects is the set of edges of $T$, and where $\Omega(T)(e_1,\dots,e_n;e)=\{\ast\}$ if and only if there exists a subtree of $T$ with leaves $\{e_1,\dots,e_n\}$ and root $e$ (necessarily unique), and it is empty otherwise. The operadic composition corresponds to \emph{grafting} of subtrees, which means successive identifications of the root of a subtree with a leaf of another. Morphisms in the category $\omg$ are defined by declaring $\Omega(-)\colon \omg\to \Op$ a fully faithful embedding.

\noindent There is a fully faithful embedding $\Delta\to \omg$, which looks like this:
\begin{center}
	\adjustbox{scale=0.7,center}{
		\begin{tikzcd}
			&                   &                            &    & {}                                                                &            \\
			\Delta \ \ni & {[2] = \{0<1<2\}} & {} \arrow[r, "i", maps to] & {} & \quad \quad \bullet^{\ 0<1} \arrow[u, "0"', no head] \arrow[d, "1 ", no head] & \in \ \omg \\
			&                   &                            &    & \quad \quad \bullet^{\ 1<2} \arrow[d, "2", no head]                           &            \\
			&                   &                            &    & {}                                                                &           
	\end{tikzcd}}
	
	\vspace*{-.3cm}
	
	\captionof{figure}{The ordinal $[2]$ is realized as the linear tree with $3$ edges and $2$ unary vertices.}
\end{center}
\noindent It is customary, and we adopt this convention here as well, to denote by $\eta$ the image of the ordinal $[0]$ under the above inclusion. It is the unique tree with no vertex.

\noindent Dendroidal sets $\dsets$ are the category of presheaves on $\omg$. The restricted Yoneda embedding yields the fully faithful functor $\cN_d\colon \Op \to \dsets$, called the \emph{dendroidal nerve functor}. Explicitly, we have $\cN_d (P)\coloneqq  \Op(\Omega(-), P)$.

\noindent The diagram on the left commutes, making the one the right commute as well. 
\begin{center}
\begin{minipage}{0.45 \textwidth}
	\[
	\begin{tikzcd}
		\Cat \arrow[d] & \Delta \arrow[l] \arrow[d] \\
		\Op            & \omg \arrow[l]            
	\end{tikzcd}
	\]
\end{minipage}
\hfil
\begin{minipage}{0.45 \textwidth}
	\[
		\begin{tikzcd}
		\Cat \arrow[d] \arrow[r, "\cN"] & \ssets \arrow[d, "i_!"] \\
		\Op \arrow[r, "\cN_d"']          & \dsets            \ ,  
	\end{tikzcd}
	\]
\end{minipage}
\end{center}
\noindent where $i_!$ is the left Kan extension of $i\colon \Delta\to \omg$ along the Yoneda embedding.

\noindent There is a natural isomorphism $\Delta\simeq \omg/\eta$, which yields the isomorphism $\dsets/\eta\simeq \ssets$, under which the inclusion $i_!$ is the forgetful functor $\dsets/\eta\to \dsets$. More generally, for any simplicial set $M$, one has $\dsets/i_!M \simeq \ssets/M$.

 \subsubsection{Dendroidal terminology}
 
 \noindent In a tree $T $ of $\omg$, every vertex has a single output edge and $n$ input edges for some $n\geq 0$ (the \emph{arity} of $v$). The root is the unique edge without an input vertex, and we denote it by $r_T$. The \emph{leaves} of $T$ are those edges which have no output vertex.
 A vertex is \emph{external} if its input edges are contained in the leaves of $T$ or if it is a \emph{stump}, that is, a vertex with no output edge. We denote by $C_n$ the the essentially unique tree with one vertex and $n$ leaves, and call it the \emph{$n$-corolla}. Let $\eta$ be the unique tree with no vertex.
 
  \noindent Denote by $E(T)$ the set of edges of $T$, the objects of the operad $\Omega(T)$. We rephrase the poset structure on $E(T)$ as follows: for two edges $e,f$, one has $e \leq f $ if and only if the (unique) path from $e$ to the root of $T$ contains $f$. In particular, the root $r_T$ is the unique maximal element, while the minimal elements are the leaves of $T$ and the input edges of stumps.

\noindent We say that $S$ is a subtree of $T$ if it can be obtained from $T$ by successively pruning away external vertices and the outer edges attached to them from $T$. 

\noindent Given two trees $R,S$ and a leaf $\ell$ of $S$, the \emph{grafting} of  $R$ onto $S$ along $\ell$ is the tree $S\cup_\ell R$ given by the pushout
 \begin{center}
\begin{tikzcd}
	\eta \arrow[d, "r_R"'] \arrow[r, "\ell"] & S \arrow[d]  \\
	R \arrow[r]                              & S\cup_\ell R
\end{tikzcd}
 \end{center}
 where both arrows $R\to S\cup_\ell R \leftarrow S$ are inclusions of subtrees.
 
 \noindent Any tree $T$ is ismorphic to a grafting of subtrees, the indecomposable elements being the $n$-corollas and the trivial tree $\eta$.
 
\noindent For the scope of this section, we denote by $\Omega[T]$ the representable dendroidal set associated to a tree $T$. Observe that there is a natural isomorphism of dendroidal sets $$ \Omega[T]\simeq \cN_d(\Omega(T)).$$

\subsection{Tensor product of dendroidal sets}\label{tensor}

\noindent An homotopy equivalence of maps of dendroidal sets $f,g\colon X\to Y$ is modeled by considering the dendroidal tensor product with the interval (i.e. the $1$ corolla $C_1 \simeq [1]$) $h\colon X\otimes C_1 \to Y$ with the usual restrictions $h_0=f$, $h_1=g$. 

\noindent Tensor product of dendroidal sets is the bifunctor $$ \otimes \colon \dsets \times \dsets \longrightarrow\dsets$$ defined as the two variable left Kan extension along the Yoneda embeddings under the bifunctor $\omg\times\omg\longrightarrow\dsets, \quad (T,S)\mapsto\cN_d(\Omega(T)\otimes_{BV }\Omega(S)),$ where $\otimes_{BV}$ denotes the Boardman-Vogt tensor product of discrete operads (\cite{BV:HIASTS}). Restricted to simplicial sets, the tensor product is just the cartesian product.

\noindent The BV tensor product of discrete operads is characterized by the universal property that are equivalences of categories $$ \alg_Q(\alg_P(\ssets))\simeq \alg_{P\otimes_{BV} Q}(\ssets)\simeq \alg_P(\alg_Q(\ssets)),$$  for any two operads $P$, $Q$, and the product for algebras is the objectwise cartesian product. The operad $P\otimes Q$ is a quotient of $P\times Q$ under some relations, notably the \emph{interchange relation}: $$ (p\otimes y)\circ(c_1\otimes q,\dots,c_n\otimes q)= \sigma_{n,m}^* (d\otimes q)\circ(p\otimes d_1,\dots,p\otimes d_m),$$ where $\circ$ denotes the total operadic composition and the permutation $\sigma_{n,m}$ is the unique element of $\Sigma_{nm}$ making sense of the above formula.

\noindent  Let us now illustrate with an example our case of interest, the tensoring with the $1$-corolla.

\begin{example}\label{ex:bv}
Consider the operads

\adjustbox{scale=0.7,center}{
\begin{tikzcd}[column sep={1.2cm,between origins}]
	e_1 & e_2 & e_3 & \quad & 0 \arrow[d, no head] \\
	\Omega(C_3) \ = & \bullet \ p \arrow[ul, no head] \arrow[u, no head] \arrow[ur, no head] \arrow[d, no head] & 
	& \Omega(C_1) \ = & \ \circ \ q \arrow[d, no head] \\
	& r & & & 1
\end{tikzcd}
}

\noindent The Boardman-Vogt interchange relation for $\Omega(C_3)\otimes \Omega(C_1)$ yields the identification:
\vspace{0.5cm}

\adjustbox{scale=0.6,center}{
\begin{tikzcd}
	{(e_1,0)} & {(e_2,0)}                                                                                                                                    & {(e_3,0)} &   & {(e_1,0)}                                                                                                & {(e_2,0)}                                                                                              & {(e_3,0)}                                                                                               \\
	& \quad \quad \quad \bullet \ p\otimes 0 \arrow[u, no head] \arrow[ru, no head] \arrow[lu, no head] \arrow[dd, "{(r,0)}" description, no head] &           &   & \quad \quad \quad \circ \ e_1 \otimes q \arrow[u, no head] \arrow[rdd, "{(e_1,1)}" description, no head] & \quad \quad \quad \circ \ e_2\otimes q \arrow[u, no head] \arrow[dd, "{(e_2,1)}" description, no head] & \quad \quad \quad \circ \ e_3\otimes q \arrow[u, no head] \arrow[ldd, "{(e_3,1)}" description, no head] \\
	&                                                                                                                                              &           & \text{{\huge $=$}} &                                                                                                          &                                                                                                        &                                                                                                         \\
	& \quad \quad \quad \circ \ r\otimes q                                                                                                         &           &   &                                                                                                          & \quad\quad \bullet \ p\otimes 1 \arrow[d, no head]                                                     &                                                                                                         \\
	& {(r,1)} \arrow[u, no head]                                                                                                                   &           &   &                                                                                                          & {(r,1)}                                                                                                &                                                                                                        
	\end{tikzcd}} 
\vspace*{-.3cm}

\captionof{figure}{The Boardman-Vogt interchange relation for $\Omega(C_3)\otimes\Omega(C_1)$.}
\end{example}

\begin{rmk}
The reader should be aware that this tensor product does \emph{not} endow dendroidal sets with a monoidal structure, as it fails to be associative up to coherent isomorphism. The tensor product is, however, associative up to coherent homotopies \cite[\S 4.4]{HeMo:SDHT}, and endows the $\infty$-category of dendroidal $\infty$-operads with a symmetric monoidal structure, see \cite[\S 5]{HM:OELIODIO}. 
\end{rmk}

\noindent The tensor product admits a right adjoint in each variable, and we denote by $\hom(-,-)$ the underlying simplicial set of this internal hom functor, that is, $$ \hom\colon \dsets^\text{op}\times\dsets\longrightarrow\ssets \quad (X,Y)\mapsto \hom(X,Y)_\bullet \simeq  \dsets(i_!(\Delta^\bullet)\otimes X, Y).$$

\subsection{Dendroidal homotopy theory} \label{ooopds} \label{homtheory}The homotopy theory of dendroidal $\infty$-operads was defined in \cite{MW:DS} and \cite{CM:DSMHO} by extending Joyal’s theory of $\infty$-categories as quasicategories. We will not need explicit definition of the following fundamental notions, which we will treat syntetically:

\begin{itemize}
	\item Analogously to the description of the morphisms in $\Delta$ as generated by faces and degeneracies under the simplicial identities, morphisms in $\omg$ are also generated (\cite[\S 3.3.4]{HeMo:SDHT}) by dendroidal faces and degeneracies, together with isomorphisms of trees, which satisfy the \emph{dendroidal identities}. 
	
	\item Some of the notions of the homotopy theory of simplicial sets can be formulated in the context of dendroidal sets. In particular, one can talk about the \emph{boundary} of a tree and  inner and external horns of a tree. Inner horns are relative to inner edges, while external horns divide into to classes, relative respectively to leaf vertices or root vertices. Contrarily to left and right horns in simplicial sets, root and leaf horns are not dual to each other, as the category $\omg$ does not admit a self duality.
\end{itemize}

\begin{deff}\label{def:normal}
	A dendroidal set $X$ is a \emph{dendroidal $\infty$-operad} if it has the right lifting property against inner horn inclusions of trees.

	\noindent A dendroidal set $X$ is \emph{normal} if, for any tree $T$, the action of $\mathsf{Aut}(T)$ on the set $X_T$ is free.
	
	\end{deff}

\noindent Observe that, if $X\simeq \cN_d P$, then $X$ is normal if and only if $P$ is $\Sigma$-free.

\begin{theorem}[{\cite{CM:DSMHO}}] There exists a model structure on the category $\dsets$ of dendroidal sets, called the \emph{operadic model structure}, with the following properties:
	
	\begin{enumerate}
		\item The cofibrations are the normal monomorphisms.
		\item  The fibrant objects are the dendroidal $\infty$-operads.
		\item A map between normal dendroidal sets is a weak equivalence if and only if for every dendroidal $\infty$-operad $X$, the map $$ \hom(B,X)\longrightarrow\hom(A,X)$$ is a categorical equivalence of $\infty$-categories.
	\end{enumerate} Moreover, this model structure is left proper and cofibrantly generated.
\end{theorem}

\noindent The homotopy theory of dendroidal left fibrations, which will appear in \Cref{application}, was first defined in \cite{He:AOIO} and generalizes that for left fibration of simplicial sets. It models $\infty$-operads cofibred in groupoids.

\begin{deff}A morphism of dendroidal sets $p\colon Y \to X$ is a \emph{dendroidal left fibration} if it has the right lifting properties against inner and leaf horn inclusions of trees.
\end{deff}

\begin{notation}
	For a morphism of dendroidal sets $p\colon Y \to X$ and an object $x$ of $X$, we denote by $Y_x$ the fibre over $x$, that is, $$ Y_x\simeq p^{-1}(x).$$ If $p$ is a dendroidal left fibration, $Y_x$ is a Kan complex (\cite[Remark 9.60]{HeMo:SDHT}).
\end{notation}

\begin{theorem}[\cite{He:AOIO}]
	\noindent Let $X$ be a dendroidal set. The category $\dsets/X$ carries a left proper cofibrantly generated model structure, called the \emph{covariant model structure}, with the following properties:
	\begin{enumerate}
		\item The cofibrations are the normal monomorphisms over $X$.
		\item The fibrant objects are the dendroidal left fibrations over $X$.
		\item A map $(A,u)\to (B,v)$ between normal objects over $X$ is a weak equivalence if and only if for any dendroidal left fibration $(Y,p)$, the map $$ \hom_X(B,Y)\longrightarrow \hom_X(A,Y)$$ is a weak homotopy equivalence of Kan complexes, where $$ \hom_X(A,Y)= \hom(A,Y)\times_{\hom(A,X)}\{u\}.$$
		\item A map $(Y,p)\to (Z,q)$ between dendroidal left fibrations is a weak equivalence if and only if, for any object $x$ of $X$, the induced map $Y_x\to Z_x$  between the fibres over $x$ is a weak homotopy equivalence of Kan complexes.
	\end{enumerate}
	\noindent Moreover, if $f\colon X \to Y$ is a map of dendroidal sets, the adjunction $$ \adjunction{f_!}{\dsets/X}{\dsets/Y}{f^*}$$ is a Quillen adjunction with respect to the covariant model structure, which is a Quillen equivalence when $f$ is an operadic weak equivalence.
\end{theorem}

\noindent Slicing over the point, that is under the equivalence $\ssets\simeq\dsets/\eta$, the operadic model structure yields the Joyal model structure, and the covariant model structure yields the homonimous one for left fibrations over a simplicial set.

\noindent The functor $\hom(-,-)$ homotopy-enriches $\dsets$ over simplicial sets with the
Joyal model structure \cite[Prop.~9.28]{HeMo:SDHT}, and $\hom_X(-,-)$ enriches
$\dsets/X$ over the Kan--Quillen model structure \cite[Prop.~9.66]{HeMo:SDHT}.
In particular, both give models for the respective mapping spaces when computed on bifibrant objects.

\subsection{Last conventions}\label{subs:last}When in presence of the canonical fully faithful functors $i\colon \Delta\to \omg$,  $i_!\colon \ssets \to \dsets$, $\omg \to \Op \to \dsets$. In particular, given a tree $T$, we still write $T$ for the operad $\Omega(T)$ and for the dendroidal set $\Omega[T]$.  Observe that there are  isomorphisms of dendroidal sets $\eta\simeq \Delta^0$ and $C_1\simeq \Delta^1$, but that the analogy stops here, as for $n\neq 1$, there is not even a map $C_n\to \Delta^n$. We will freely use any of these functors explicitly whenever it is convenient or helps to emphasize a point.

\section{Localization of dendroidal $\infty$-operads}\label{sectionOnloc}

\noindent We now define derived localization of dendroidal sets, in a way that extends that of quasi-categories in the sense of Joyal and Lurie (\cite{Lu:HA}). We then construct a model for normal dendroidal sets in \Cref{modeloc}. For compatibility of localization of dendroidal sets with the covariant model structure we direct the reader to \Cref{section4}.

\subsection{The definition of localization}\label{defloc}
\noindent Let us start by some preliminary

\begin{deff}
	A \emph{normalization} of a dendroidal set $X$ is a trivial fibration $X'\xrightarrow{\sim} X$ in the operadic model structure, with $X'$ normal.
	
	\noindent Given a morphism of dendroidal sets $f\colon X \to Y$, a \emph{normalization} of $f$ is a commutative diagram
	 \[ 
	\begin{tikzcd}
		X' \arrow[d, "\wr"'] \arrow[r, "f'"] & Y' \arrow[d, "\wr"] \\
		X \arrow[r, "f"] & Y
	\end{tikzcd}\] where both vertical arrows are normalizations.
\end{deff}

\noindent Normalizations exist and are unique up to operadic weak equivalence. As explained in  \cite[Remark 9.21]{HeMo:SDHT}, an explicit construction of a normalization with contractible fibres is given by the projection $X\times W^* \cE\to X$, where $\cE$ is the simplicial Barratt-Eccles operad and $W^*$ is the operadic homotopy coherent nerve functor, right adjoint to $W_!$.

\begin{deff}Let $\lambda\colon X \to Y$ be a morphism between normal dendroidal sets and let $S\subseteq X_{C_1}$ be a subset of $1$-morphisms. The map $\lambda$ is a \emph{localization of $X$ at the set of morphisms $S$} if, for any dendroidal $\infty$-operad $Z$, the morphism between the simplicial hom objects $$\hom(Y,Z)\longrightarrow \hom(X,Z) $$ is fully faithful, with essential image given by those maps $X\to Z$ sending $S$ to equivalences.
	
\noindent If $\lambda\colon X \to Y$ is a morphism between non-necessarily normal dendroidal sets, we say that $\lambda$ is a localization if any, or equivalently one, normalization $\lambda'$ of $\lambda$ is. 

\end{deff}

\noindent Localization is unique up to operadic weak equivalence, and we denote by $X[S^{-1}]$ 'the' localization.

\noindent  Observe that, if we denote by $\hom_S(X,Z)$ the full sub simplicial set of $\hom(X,Z)$ spanned by those maps $X\to Z$ sending $S$ to equivalences in $Z$, the universal property of the localization allows to identify $\hom(X[S^{-1}],Z)$ with $\hom_S(X,Z)$.

\begin{rmk}
	It is immediate to see that the definition applied to a morphism of dendroidal sets recovers the localization of quasi-categories in the sense of  \cite[Definition 1.3.4.1]{Lu:HA}.
\end{rmk}

\noindent We can construct an explicit model for the localization of a normal dendroidal set. 
 
\begin{prop}\label{modeloc} 
Denote by $J$ be the nerve of the connected groupoid on two objects. 
Given a normal dendroidal set $X$ and a subset $S\subseteq X_{C_1}$ of $1$-morphisms in $X$, the localization of $X$ at $S$ is realized by the map $\lambda\colon X\to \cL(X,S)$ defined by the pushout diagram
\begin{center}
\begin{equation}\label{local}
\begin{tikzcd}
{\underset{s\in S}{\bigsqcup} C_1} \arrow[d] \arrow[r] & X \arrow[d, "\lambda"] \\
\underset{s\in S}{\bigsqcup} J \arrow[r]                       & {\cL(X,S)}    \ .
\end{tikzcd}
\end{equation}
\end{center}
\end{prop}

\begin{proof}	Observe that, by left properness of the operadic model structure, the pushout in \Cref{local} is a homotopy pushout. Let $Z$ be a dendroidal $\infty$-operad and consider the map $$\lambda^*\colon \hom(\cL(X,S),Z)\longrightarrow \hom(X,Z).$$ \noindent The essential image of $\lambda^*$ consists of all functors sending $S$ to equivalences, so we just need to show $\lambda$ is also fully faithful. 
To this end, it is sufficient to show that the diagram
\begin{center}
\begin{tikzcd}
{\hom({\cL(X,S)\otimes C_1}, Z)} \arrow[d] \arrow[r] & {\hom({X\otimes C_1},Z)} \arrow[d] \\
{\hom({\cL(X,S)\otimes \partial C_1},Z)} \arrow[r]  & {\hom({X\otimes \partial C_1},Z)} 
\end{tikzcd}
\end{center}
is a homotopy pullback for the Joyal model structure, which happens if the diagram

\begin{equation}\label{hpushout}
\begin{tikzcd}
{X\otimes \partial C_1} \arrow[r] \arrow[d] & {X\otimes C_1} \arrow[d] \\
{\cL(X,S)\otimes \partial C_1} \arrow[r]      & {\cL(X,S)\otimes C_1}     
\end{tikzcd}
\end{equation} is a homotopy pushout for the operadic model structure. Since the diagram in \eqref{local} is a transfinite composition of homotopy pushouts along the morphisms $\{C_1\xrightarrow{s} J\}_{s\in S}$, it is sufficient to show that it is an homotopy pushout just in the case when $S=\{s\}$, where it appears as the front face of the commutative cube

\begin{center}
\begin{tikzcd}
{C_1\otimes \partial C_1} \arrow[rr] \arrow[dd] \arrow[rd, "s\otimes \id"] &                                                      & {C_1\otimes C_1} \arrow[dd] \arrow[rd, "s\otimes \id"] &                                   \\
                                                                          & {X\otimes \partial C_1} \arrow[rr] \arrow[dd] &                                                        & {X\otimes C_1} \arrow[dd] \\
{J\otimes \partial C_1} \arrow[rd] \arrow[rr]                   &                                                      & {J\otimes C_1} \arrow[rd]                   &                                   \\
                                                                          & {\cL(X,S)\otimes\partial C_1} \arrow[rr]    &                                                        & {\cL(X,S) \otimes C_1}   
\end{tikzcd}
\end{center}

\noindent The front face of the cube is a homotopy pushout because all the others are, and this concludes the proof.

\end{proof}

\noindent The above construction essentially means that to localize a dendroidal set at a set of $1$-morphisms, we can first localize its underlying simplicial set and then glue it to the original dendroidal set $X$, as we explain in the next

\begin{rmk}
Let $X$ be a normal dendroidal set, and write $M\coloneqq i^*X$ for its underlying simplicial set. Of course, we have $S\subseteq X_{C_1}=M_1$, so we can localize $M$ at $S$. Let ${\overline{\lambda}\colon M \to \cL(M,S)}$ be the localization map. The localization of $X$ at $S$ can be realized as the following homotopy pushout
\begin{center}
\begin{tikzcd}
M=i^*X \arrow[d, "\overline{\lambda}"] \arrow[r, "\epsilon_X"] & X \arrow[d, "\lambda"] \\
{ \cL(M,S)} \arrow[r]                              & {\cL(X,S)}           
\end{tikzcd}
\end{center}
where the top horizontal map is the counit of the adjunction $(i_!,i^*)$.

\end{rmk}

\noindent Given two endofunctors of $X$ which preserve $S$, there is an easy way to check when they are homotopy equivalent as endofunctors of the localization $X[S^{-1}]$. Recall the following

\begin{deff} Let $f,g\colon X \to Y$ be two maps of dendroidal sets.
	An \emph{homotopy} between $f$ and $g$ is a morphism $ h\colon X\otimes C_1\to Y$ in $\dsets$ with the property that $h_0=f$ and $h_1=g$, where $h_i$ is the restriction of $h$ along the leaf, resp. root, inclusions $\eta\xrightarrow{\{i\}}C_1$, $i=0$, resp. $i=1$. 
\end{deff}

\noindent For $x\in X_\eta$, the arrow $h_x \colon f(x)\to g(x)$ in $Y_{C_1}$ is called a \emph{component} of $h$. 
\begin{lemma}\label{lemmaequiv}
Let $X$ be normal, $S\subseteq X_{C_1}$ a set of $1$-morphisms. Let $f,g\colon X \to X$ be two maps sending $S$ to itself. If $f$ and $g$ are homotopic via a homotopy whose components are in S, then $f$ and $g$ are homotopy equivalent as maps $f,g\colon X[S^{-1}]\to X[S^{-1}]$. \end{lemma}
\begin{proof} Let $h\colon X \otimes C_1\to X$ be the homotopy between $f$ and $g$ whose components lie in $S$. The transpose of $h^*\colon \hom(X,Z)\to \hom(X\otimes C_1,Z)\simeq \hom(C_1,\hom(X,Z))$ induces a morphism of simplicial sets $ \hom_S(X,Z)\times \Delta^1\to\hom(X,Z)$. By the hypotheses on $f$ and $g$ and fullness of  $\hom_S(X,Z)$ in $\hom(X,Z)$, there is an induced homotopy $\mathfrak{h}\colon \hom_S(X,Z) \times \Delta^1 \to \hom_S(X,Z)$. To prove that $\mathfrak{h}$ is a natural equivalence, it suffices to see that, for any object $\phi $ of $\hom_S(X,Z)$, the arrow $\mathfrak{h}_\phi\colon f^*(\phi)\to g^*(\phi)$ is an equivalence in $\hom_S(X,Z)$, that is in $\hom(X,Z)$. So we only need to check that for any $x\in X_\eta$, the $1$-morphism $(\mathfrak{h}_\phi)_x\colon \phi(f(x))\to \phi(g(x))$ is an equivalence in $Z$. As $(\mathfrak{h}_\phi)_x=\phi(h_x)$, and since by hypothesis $\phi(h_x)$ is a weak equivalence as $h_x$ is an arrow in $S$, we have that $\mathfrak{h}$ establishes an equivalence between $f^*$ and $g^*$. This concludes the proof.
\end{proof}

\noindent The homotopy theory of algebras over a dendroidal $\infty$-operad $X$ is governed by the covariant model structure on the over-category $\dsets/X$, so it  is a natural question to investigate compatibility of the localization of dendroidal $\infty$-operads with the covariant model structure for dendroidal left fibrations. We study this in \Cref{section4} (\Cref{localg}).
\section{The delocalization theorem}\label{sec:rootfunctor}

We will now functorially associate to every dendroidal set $X$ a discrete operad $(\omg/X)^\otimes$. To mimic Joyal's theorem for simplicial sets, its objects should be the representable elements of the presheaf $X$, so we will call the operad $(\omg/X)^\otimes$ the \emph{operad of elements} of $X$.

\subsection{The operad of elements and the root functor}\label{thedef} 

\noindent For every dendroidal set $X$, disjoint union endows the overcategory $\dsets_{/X}$ with a symmetric monoidal structure, with associated operad $(\dsets_{/X})^\sqcup$. A first way to define the operad of elements of $X$ would be to do it as the full suboperad $(\omg/X)^\sqcup$ of $(\dsets_{/X})^\sqcup$ spanned by the representable objects. As we will explain in Remark \ref{rmk:nogood}, this operad can't work. What we do now is constructing what we will prove to be the \emph{maximal} wide suboperad of $(\omg/X)^\sqcup$ which works. To give its defininition (Definition \ref{operadofdendrices}), we need to introduce the following notions.

\begin{deff}
	Given trees $T_i, S$, a morphism $f\colon \bigsqcup_{i=1}^n T_i \longrightarrow S $, $f=\bigsqcup f_i$, is called \emph{independent} if the edges $f(r{_{T_i}})$ and $f(r_{T_{i'}})$ are incomparable in the poset $E(S)$ whenever $i\neq i'$.
	  
\end{deff}

\noindent Observe that this implies that for all edges $e\in E(T_i)$, $e'\in E(T_{i'})$, the edges $f_i(e)$ and $f_{i'}(e')$ are incomparable in $E(S)$. Also, the independency condition is trivially satisfied when $n=1$.

\begin{deff}[{\cite{HeMo:SDHT}}]
	A $f\colon \bigsqcup_{i=1}^n T_i \longrightarrow S $ is called \emph{wide} if any maximal monotonic path in the poset $E(S)$, that is from a minimal element to the root $r_S$, contains one element of the form $f(r_{T_i})$ for some $i$.
\end{deff} 

\noindent Observe that if $f$ is also independent, if there exists such a constituent $T$, then it is unique.

\noindent For independent maps, we can reformulate the wideness condition in the following useful way.
\begin{lemma}\label{wide}
Let $f\colon \bigsqcup_{i=1}^n T_i \longrightarrow S  $ be an independent map. Then $f$ is wide if and only if ${S(f(r_{T_1}), \dots, f(r_{T_n}); r_S)\neq \emptyset}$.
\end{lemma}
\begin{proof} It suffices to observe that, given edges $e_1,\dots,e_n$ of $S$, the only obstruction to the existence of a subtree with leaves $e_1,\dots,e_n$ and root $r_S$ is the existence of a maximal monotonic path $\mathfrak{p}=\{l<a_1<\dots<a_m<r_S\}$ in the poset of edges of $S$ such that $\{e_1,\dots,e_n\}\cap \mathfrak{p}=\emptyset$. 
\end{proof}

\noindent Let us single out an important class of maps.
\begin{deff}\label{dec2}
A morphism of trees $f\colon T \to S$ is \emph{root preserving} if $f(r_T)=r_S$. 
\end{deff}

\noindent Any root preserving map is wide. The converse is true if the root vertex of $S$ is not unary or when $S$ does not have nullary vertices. In particular, observe that all maps of linear trees $[n]\to [m]$ are wide and independent. In fact, wide independent maps are generated, under composition and coproduct, by the root preserving morphisms and the forest root faces, which we introduce with the next

\begin{example}	Grafting of unordered uples of trees defines a symmetric functor $$ (\omg^{\times n})_{\Sigma_n}\longrightarrow \omg \, \quad (T_1,\dots,T_n)\mapsto C_n\circ (T_1,\dots,T_n), $$ where $C_n\circ (T_1,\dots,T_n)$ is the tree obtained by grafting each $T_i$ on one leaf of the $n$-corolla $C_n$, $$\overline{F}=C_n\circ (T_1,\dots,T_n) .$$ The inclusions induce a map $$ T_1\sqcup\dots \sqcup T_n \longrightarrow C_n\circ (T_1,\dots,T_n),$$ which is both wide and independent. We refer to this map as a \emph{forest root face} (this terminology appears in \cite[\S 6.6]{HeMo:SDHT}).
	
	\begin{center}
		\begin{tikzcd}[sep=small]
			&                                                                   &        &                             & {} \arrow[rd, no head] &                                                 & {}                                             &                    &    &                        &                                                                     &    & {}                                              &                                                                     & {}                    &                        \\
			{} \arrow[rd, no head] & {}                                                                & {}     & \bullet \arrow[dd, no head] &                        & \bullet \arrow[ru, no head] \arrow[rd, no head] & {} \arrow[d, no head]                          & {}                 &    & {} \arrow[rd, no head] & {}                                                                  & {} & \bullet \arrow[dd, no head]                     & \bullet \arrow[rd, no head] \arrow[lu, no head] \arrow[ru, no head] & {} \arrow[d, no head] & {} \arrow[ld, no head] \\
			& \bullet \arrow[u, no head] \arrow[ru, no head] \arrow[d, no head] & \oplus &                             & \oplus                 &                                                 & \bullet \arrow[ru, no head] \arrow[d, no head] & {} \arrow[r, hook] & {} &                        & \bullet \arrow[u, no head] \arrow[ru, no head] \arrow[rrd, no head] &    &                                                 &                                                                     & \bullet               &                        \\
			& {}                                                                &        & {}                          &                        &                                                 & {}                                             &                    &    &                        &                                                                     &    & \bullet \arrow[rru, no head] \arrow[d, no head] &                                                                     &                       &                        \\
			&                                                                   &        &                             &                        &                                                 &                                                &                    &    &                        &                                                                     &    & {}                                              &                                                                     &                       &                       
		\end{tikzcd}
		
		\vspace*{-.3cm}
		
		\captionof{figure}{A forest root face.}

	\end{center}
\end{example}

\noindent We can now define the operad of elements.

\begin{deff}\label{operadofdendrices}
	Let $X$ be a dendroidal set. Its \emph{operad of elements} $(\omg/X)^\otimes$ is the sub operad of $(\dsets_{/X})^\sqcup$ specified by the following data:
	\begin{itemize}
		\item the set of objects of $(\omg/X)^\otimes$ is the set of elements of $X$ as a presheaf, that is, $$ \mathsf{Ob}(\omg/X)^\otimes=\{ (T,\alpha) \ | \ T \in \omg \text{ and } \alpha\colon T \to X\};$$
		\item for objects  $(S_1,\alpha_1),\dots,(S_n,\alpha_n)$, $(R,\beta)$, the set of operations $$(\omg/X)^\otimes((S_1,\alpha_1),\dots,(S_n,\alpha_n);(R,\beta))$$ is given by the wide and independent maps $f=(f_i)_i\colon S_1\sqcup \dots \sqcup S_n \to R$ in $\dsets_{/X}$.

	\end{itemize}
\end{deff}

\begin{rmk}\label{rmk:sfree}We list some properties of the operad of elements.
	\begin{enumerate}
		\item For a general $X$, the underlying category of $(\omg/X)^\otimes$ is a wide but not necessarily full subcategory of the category of elements of $X$, which is related to the presence of units (cfr the discussion after Definition \ref{dec2}).
		\item In particular, if $X$ is a simplicial set than $(\omg/X)^\otimes $ \emph{is} the category of elements $\Delta/X$. Its objects are the pairs $([n],f)$, with $n\geq 0$ and $f\colon \Delta^n \to X$; a morphism $F\colon ([n],f)\to ([m],g)$ is just a map of representable simplicial sets $F\colon [n]\to [m]$ over $X$. The choice of embedding $\Delta\to \omg$ chosen in Section \ref{section1} identifies the root of a simplex $[n]$ with $n$ , hence root preserving morphisms of linear trees are just the $F\colon [n]\to[m]$ preserving the last vertex, that is$F(n)=m$.
	\item By the independency condition on operations, the operad $(\omg/X)^\otimes$ is $\Sigma$-free.
	\end{enumerate}
\end{rmk}

\noindent The construction $X\mapsto (\omg/X^\otimes$ is functorial in $X$ via postcomposition, and we get an endofunctor \[\cN_d(\omg/-)^\otimes\colon \dsets \to \dsets.\] Let us postpone to later the proof of the next central

\begin{lemma}\label{lke}The functor $\cN_d(\omg/-)^\otimes$ is cocontinuous. Moreover, it preserves normal monomorphisms of dendroidal sets.
\end{lemma}

\noindent In other words, it means that $\cN_d(\omg/-)^\otimes$ is equivalent to the left Kan extension of its restriction to representables, and moreover in an homotopy-coherent way

\noindent Now, let $T$ be a tree and $(S,\alpha)$ an object of $(\omg/T)^\otimes$. Evaluation of $\alpha$ at the root of $S$ yields an assignment \begin{equation}\label{evaluation}
	\mathsf{Ob}(\omg/T)^\otimes\ni (S,\alpha)\mapsto \alpha(r_S) \in \mathsf{Ob}(T)=E(T).
\end{equation}

\begin{prop}\label{prop:extension}
	The assignment in \Cref{evaluation} extends to a map of operads $$ \scriptr_T\colon (\omg/T)^\otimes \longrightarrow T$$ that we call the \emph{root functor} for $T$.
\end{prop}
\begin{proof}
The set of operations of a tree $T$ is either empty or a singleton. It follows that to prove the statement we just need to check that, if there exists a wide independent map $f\colon (S_1,\alpha_1), \dots, (S_n,\alpha_n)\to (R,\beta)$, then ${T(\alpha_1(r_{S_1}),\dots, \alpha_n(r_{S_n}); \beta(r_R))\neq \emptyset}$. By Lemma \ref{wide}, $R(f(r_{S_1}),\dots,f(r_{S_n}); r_R)\neq \emptyset$, and since $\beta$ is a map of operads  and $\beta \circ f =(\alpha_1,\dots,\alpha_n)$ we have the thesis, as wanted.
\end{proof}

\begin{rmk}\label{rmk:nogood}
	The assignment in \Cref{evaluation} makes sense even for the \emph{full} suboperad $(\omg/X)^\sqcup\subseteq (\dsets/X)^\sqcup$. However, Proposition \ref{prop:extension} does \emph{not} hold for $(\omg/X)^\sqcup$. An easy way to see this is taking $X=C_n$ any $n$-corolla, and the binary operation $(\id,\id)\colon (C_n,\id), (C_n,\id) \to (C_n,\id)$. A candidate image for this latter has to be a subtree of $C_n$ with leaves $(r_{C_n},r_{C_n})$ and root $r_{C_n}$, but such subtree does not exist. Proposition \ref{prop:extension} characterizes $(\omg/X)^\otimes$ as the maximal sub operad of $(\omg/X)^\sqcup$ such that the statement holds: maximality consists in observing that the condition shown in the proof is an if and only if.
\end{rmk}

\noindent Because of cocontinuity of $\cN_d(\omg/-)^\otimes$, we can extend the root functor to every dendroidal set.

\begin{deff}\label{rootfnc}
Let $X$ be a dendroidal set. The \emph{root functor} of $X$ is the morphism of dendroidal sets $$\scriptr_X \colon \cN_d(\omg/X)^\otimes\longrightarrow X$$ defined as the colimit $$\scriptr_X \coloneqq \underset{T\to X}{\colim} \scriptr_T\colon \cN_d(\omg/T)^\otimes\longrightarrow T.$$

\end{deff}

\noindent We can interpret an object $(S,\alpha)$ in $\cN_d(\omg/X)^\otimes$ as a decoration of the tree $S$ with objects and operations in $X$, together with all higher coherences governing operadic composition (which are not visible in the picture). In particular, by contracting all inner edges of $S$ one obtains an operation of $X$. The root functor then sends $(S,\alpha)$ to the target of this operation.  Visually

\vspace{-0.4cm}
\begin{figure}[h]
	\centering
	\adjustbox{scale=0.8,center}{
		\begin{tikzcd}
			& \quad                                                                             &                       & \quad &             \\
			\alpha(\ell_1) \arrow[rd, no head] & \alpha(\ell_2)                                                                    & \alpha(\ell_3)        &       &             \\
			& \ \ \bullet \ \alpha(p) \arrow[ru, no head] \arrow[u, no head] \arrow[d, no head] & {} \arrow[r, maps to, "\scriptr_X"] & {}    & \alpha(r_S) \\
			& \alpha(r_S)                                                                       &                       &       &            
			\end{tikzcd}}
		\captionof{figure}{The root functor for $S\simeq C_3$ and a map $\alpha\colon C_3\to X$.}
\end{figure}

\noindent Observe that, for any tree $T$ and morphism $\alpha\colon T \to X$, the following diagram commutes:
\begin{center}
	\begin{tikzcd}
		(\omg/T)^\otimes \arrow[r, "(\omg/\alpha)^\otimes"] \arrow[d, "\scriptr_T"'] & (\omg/X )^\otimes \arrow[d, "\scriptr_X"] \\
		T \arrow[r, "\alpha"] & X
	\end{tikzcd}
\end{center}

\begin{rmk}
	If $X$ is a simplicial set, the root functor coincides with the \emph{last vertex functor} $$\scriptr_X\colon \Delta/X \to X, \quad \scriptr_X([n],f)= f(n).$$
\end{rmk}

\noindent We conclude this section with the promised proof of cocontinuity of the nerve operad of elements construction.

\begin{proof}[{Proof (of Lemma \ref{lke})}]
	We denote by $\theta$ the restriction of $\cN_d(\omg/-)^\otimes$ to the representables, that is $\theta=\cN_d(\omg/-)^\otimes_{|_\omg}$, and we write $\widehat{\theta}$ for its left Kan extension along the Yoneda embedding. 
	
	\noindent Given a dendroidal set $X$ and a tree $T$, the set $\widehat{\theta}(X)_T$ may be described as $$ \widehat{\theta}(X)_T= \{ ((Q,x), (T,u)) \ | \ Q\in \omg, x\colon Q \to X, u\colon T \to \theta(Q)\}/\sim,$$ where, for any $\alpha\colon S \to Q$, one identifies $$ ((Q,x), (R,\alpha_* u))\sim ((S, x\alpha), (R,u)),$$ where $\alpha_*=\cN_d(\omg/\alpha)^\otimes$.
	Functoriality in $T$ is given by letting faces and degeneracies act on the second component, while functoriality in $X$ is obtained via the first component. 
	
	\noindent We construct a natural equivalence
	$$\psi_X\colon \widehat{\theta}(X)\to \cN_d(\omg/X)^\otimes$$ by defining its components $(\psi_X)_T$ by induction on the number of vertices of the tree $T$. 
	When $T\simeq \eta$, we set $$ \psi_X((Q,x), (\eta, u))\coloneqq x_*\circ u \colon \eta \to \cN_d(\omg/X)^\otimes.$$ The assignment respects the equivalence relation, as (with the same notations as above) one has $$ \psi_X((Q,x),(\eta,\alpha_* u))= x_*(\alpha_* u)= (x\alpha)_* u = \psi_X((S, x\alpha), (\eta,u)),$$ and it is straightforward to see that it induces a bijection. 
	
	\noindent Similarly, given a $n$-corolla $C_n$, $n\geq 0$, and an element $((S, x), (C_n,u)),$ we set $$ \psi_X((S, x), (C_n,u))\coloneqq x_*u \colon C_n \to \cN_d(\omg/X)^\otimes,$$ and it is the same calculation which shows that it descends to the quotient and is a bijection. 
	
	\noindent For the inductive step, let us consider a tree $T$ with at least two vertices and decompose it as the grafting $T=R\cup_a S$, with $R,S\neq \eta$. For any dendroidal set $Y$, there is a natural isomorphism $$ \cN_d(\omg/Y)^\otimes_{T}\simeq \cN_d(\omg/Y)^\otimes_R\underset{\cN_d(\omg/Y)^\otimes_\eta}{\times} \cN_d(\omg/Y)^\otimes_S.$$ The isomorphism is compatible with the equivalence relation defining $\widehat{\theta}(X)$, which means that $\widehat{\theta}(X)$ satisfies the same strict Segal condition, that is, there is a natural isomorphism $$ \widehat{\theta}(X)_T \simeq \widehat{\theta}(X)_R \underset{\widehat{\theta}(X)_\eta}{\times}\widehat{\theta}(X)_S.$$ One defines the map $(\psi_X)_T$ as $$(\psi_X)_T \coloneqq (\psi_X)_R\underset{(\psi_X)_\eta}{\times} (\psi_X)_S.$$It respects the equivalence relation and is a bijection, so we only need to check that $\psi_X$ is well defined. This follows from the fact that any tree $T$ decomposes as the grafting of corollas, and the decomposition is unique up to isomorphism and operadic associativity relations, and the Segal isomorphism is compatible with these latter, which shows that the definition of $(\psi_X)_T$ does not depend on the decomposition of $T$, as wanted.
\end{proof}

\subsection{The delocalization result}\label{section3}\label{rootpresequiv}
In Definition \ref{dec2} we introduced root preserving morphisms of trees, which are in particular wide and independent maps, hence morphisms in the underlying category of $(\omg/X)^\otimes$. Denote by $\cR_X$ the set of root preserving morphisms of $(\omg/X)^\otimes$.

\noindent If $f\colon (S,\alpha)\to (R,\beta)$ is root preserving, then $\scriptr(f)= \id_{f(r_S)}$, so the root functor factors via the localization of $\cN_d(\omg/X)^\otimes$ at $\cR_X$, as 

\[
\begin{tikzcd}
	\cN_d(\omg/X)^\otimes \arrow[rr, "\scriptr_X"] \arrow[rd, "\lambda"']& & X \\
	& \cN_d(\omg/X)^\otimes[\cR_{X}^{-1}] \arrow[ru, "\overline{\scriptr_X}"'] \ .& 
\end{tikzcd}
\]

\begin{theorem}\label{main}
 For any normal dendroidal set $X$, the root functor induces an operadic weak equivalence of dendroidal sets $$\overline{\scriptr_X} \colon \cN_d(\omg/X)^\otimes[\cR^{-1}_X]\xrightarrow{\sim} X.$$
 
 \noindent Moreover, if $X$ is a simplicial set, the root functor coincides with the last vertex functor  and yields a weak categorical equivalence $\cN(\Delta/X)[\cR_X^{-1}]\xlongrightarrow{\sim} X$
 \end{theorem}

\begin{proof} We model localization of dendroidal sets via the pushout, as illustrated in Proposition \ref{modeloc}. By Lemma \ref{lke}, localizing the nerve of $(\omg/X)^\otimes$ at $\cR_X$ is a cocontinuous endofunctor of dendroidal sets and preserves normal monomorphisms. We may therefore proceed by skeletal filtration, reducing the proof to the case of $X\simeq T$ a tree, that is a representable dendroidal set. The root functor is the nerve of the map of operads $\scriptr\colon (\omg/T)^\otimes\to T$, and we construct a section  $\mathfrak{l}_T$ of $\scriptr_T$, $$\mathfrak{l}_T\colon T \to (\omg/T)^\otimes$$ and a homotopy $$h\colon  \cN_d(\omg/T)^\otimes\otimes C_1\to \cN_d(\omg/T)^\otimes$$ between $ \mathfrak{l}_T\circ \scriptr_T$ and $\id_{(\omg/T)^\otimes}$, such that its components are root preserving morphisms. After Lemma \ref{lemmaequiv}, this implies that $h$ becomes an equivalence between $ \mathfrak{l}_T\circ \scriptr_T $ and $\id_T$ after localizing at $\cR_T$, which means that $\mathfrak{l}_T$ and $\scriptr_T$ are homotopy inverses of the other once localized.

\noindent To construct $h$, consider an edge $e$ of $T$, and let $T_e^{\uparrow}$ be the biggest subtree of $T$ having $e$ as root. We denote by $\iota_e\colon T^\uparrow_{e}\hookrightarrow T$ the associated subtree inclusion. Observe that, if $e\leq f$, then $T_e$ is a subtree of $T_f$. 

\noindent We define the morphism $\mathfrak{l}_T$ on an edge $e$ of $T$ as $$\mathfrak{l}_T(e)\coloneqq  (T^\uparrow_{e},\iota_e)\in \mathsf{Ob}(\omg/T)^\otimes.$$ 
Given edges $e_1,\dots,e_n,e$ such that $T(e_1,\dots,e_n;e)=\{*\}$, the map $$\mathfrak{l}_T \colon \{*\} \xlongrightarrow{} (\omg/T)^\otimes((T^\uparrow_{e_1},\iota_{e_1}),\dots,(T^\uparrow_{e_n},\iota_{e_n});(T^\uparrow_{e}, \iota_{e}))$$ selects the operation 
defined as the composition $$\bigsqcup_{i=1}^n{T_{e_i}}^{\uparrow}\xrightarrow{} C_n\circ (T_{e_1}^{\uparrow},\dots, T_{e_n}^{\uparrow}) \xrightarrow{\partial} T^\uparrow_e$$ of a forest root face and inner face $\partial$ sending each $T_{e_i}^{\uparrow}$ to itself and the new root vertex to the subtree $T_e^{\underline{e}}$. As the operadic composition for $T$ is the grafting of subtrees, we can check that it is a well defined map of operads $\mathfrak{l}_T\colon T\to (\omg/T)^\otimes$. 
One checks that  $\scriptr_T\circ \mathfrak{l}_T = \id_T$; for the other composition $\mathfrak{l}_T\circ \scriptr_T$, consider an object $(S,\alpha)$ in $(\omg/T)^\otimes$. We have that
  
  $$ \mathfrak{l}_T\circ \scriptr_T(S,\alpha)= \mathfrak{l}_T(\alpha(r_S)))=(T_{\alpha(r_S)}^\uparrow,\iota_{\alpha(r_S)})\ .$$
  
\noindent Since the image of $\alpha$ is contained in the subtree $T_{\alpha(r_S)}^\uparrow$, we can always write $\alpha$ as a composition 
\begin{center}
\begin{tikzcd}
S \arrow[rr, "\alpha"] \arrow[rd, "h_{(S,\alpha)}"'] &                                                                     & T \\
                                                     & T^{\uparrow}_{\alpha(r_S)} \arrow[ru, "\iota_{\alpha(r_S)}"', hook] &  
\end{tikzcd}
\end{center} 

\noindent for an unique root preserving morphism $h_{(S,\alpha)}$. In particular, $h_{(S,\alpha)}$ is a morphism in $(\omg/T)^\otimes$ and we obtain  the collection of morphisms $$h\coloneqq \{h_{(S,\alpha)} \colon (S,\alpha)\to (T_{\alpha(r_S)},\iota_{\alpha(r_S)})\}_{(S,\alpha) \in \mathsf{Ob}(\omg/T)}.$$
Let us show that $h$ defines an homotopy between $\id_{(\omg/T)^\otimes}$ and $\mathfrak{l}_T \circ \scriptr_T$. For this purpose, we need to check the Boardman-Vogt interchange relation, so consider objects  $(S_1,\beta_1),\dots,(S_n,\beta_n),(R,\gamma)$ and an operation $f \in (\omg/T)^\otimes((S_1,\beta_1),\dots,(S_n,\beta_n);(R,\gamma))$. Since $\gamma\circ f=(\beta_1,\dots,\beta_n)$, we have that 
\[ \gamma\circ f = (\iota_{\scriptr_T(\beta_1)}, \dots, \iota_{\scriptr_T(\beta_n)})\circ (h_{(S_1,\beta_1)}, \dots, h_{(S_n,\beta_n)}),\] which is precisely the wanted relation. Consider the functor of dendroidal set given by the dendroidal nerve of $h$; precomposing $\cN_d(h)$ with the natural map $\cN_d(\omg/T)^\otimes\otimes C_1\to \cN_d((\omg/T )^\otimes\otimes C_1)$, we obtain an homotopy between between the identity of $(\omg/T)^\otimes$ and the composition $\mathfrak{l}_T\circ \scriptr_T$. As the components of this homotopy are root preserving morphisms, it becomes an equivalence after localizing at $\cR_T$, and this concludes the argument for a general dendroidal set $X$.

\noindent The specialization to the case of $X$ simplicial sets, that is Joyal's delocalization theorem, follows from Remark \ref[(2)]{rmk:sfree} and the fact that the operadic model structure reduces to Joyal model structure on simplicial sets.
\end{proof}

\noindent The operadic model structure on $\dsets$ admits a left Bousfield localization whose homotopy category is equivalent to that of group-like $\mathbb{E}_\infty$-algebras, in turn equivalent to infinite loop spaces. It is called the \emph{stable model structure}, and was first introduced by Bašić-Nikolaus in \cite{BN:DSMCS}. It has the property that the induced model structure on the overcategory $\ssets\simeq \dsets/\eta$ coincides with the Kan-Quillen model structure. In particular, one localizes also by the arrow $C_1\to J$, which becomes an equivalence, so from \Cref{main} we deduce the following
\begin{coro}\label{westable}
For any normal dendroidal set $X$, the root functor $\scriptr_X \colon \cN(\omg/X)^\otimes\to X$ is a weak equivalence in the stable model structure for $\dsets$.

\noindent In particular, when $X$ is a simplicial set this gives an equivalence of weak homotopy types $\cN(\Delta/X)\xrightarrow{\sim}X$.
\end{coro}

\begin{rmk}\label{rmk:smcat}
When $X$ is a symmetric monoidal $\infty$-category, \cite{A:MRCMMIC} guarantees
the existence of a discrete symmetric monoidal category localizing at $X$. The question becomes whether such an equivalence can still be realized via the root functor, notably via its adjoint map $\cN(\mathsf{Env}(\omg/X))^\otimes\to X,$ a functor of symmetric monoidal $\infty$-categories from the envelope of $(\omg/X)^\otimes$ (still discrete), into $X$. 

\noindent In fact, it is immediate to see that the operad $(\Omega/X)^\otimes$
does not inherit a symmetric monoidal structure when $X$ is symmetric monoidal itself. However, it can be seen that it naturally carries a
\emph{pre-coCartesian} structure in the sense of \cite{KSW:ACFA}, and in ongoing work with M. Sporring we show that this structure can be exploited to prove that in this case $X$ is equivalent, as a symmetric monoidal $\infty$-category, to a localization of the envelope $\mathsf{Env}(\omg/X)^\otimes$ (at a class of morphisms strictly larger than that of root preserving ones.)
\end{rmk}

\section{Localization and un/straightening equivalences}\label{application}\label{sec:loccostalg} 

\noindent In this section, we describe the homotopy theory of algebras over a dendroidal $\infty$-operad in terms of locally constant algebras over its operad of elements. To this end, we study the compatibility of dendroidal localization with the covariant model structure for dendroidal left fibrations, appearing in the un/straightening equivalence.

	\subsection{Localization and the covariant model structure}\label{section4}
	
	\noindent For a dendroidal set $X$ and any subset of morphisms $S\subseteq X_{C_1}$, the localization map $\lambda \colon X \to X[S^{-1}]$ induces an adjunction of over-categories $$ \adjunction{\lambda_!}{\dsets/X}{\dsets/X[S^{-1}]}{\lambda^*},$$ which is a Quillen adjunction with respect to the covariant model structure for dendroidal left fibrations (\cite[Proposition 2.4]{He:AOIO}).
	\noindent It is natural to expect dendroidal left fibrations over the localization to be weakly equivalent to those left fibrations over $X$ for which all the maps of fibres $f_!\colon Y_a\to Y_b$ induced by the morphisms $f\colon a \to b$ in $S$ are weak homotopy equivalences of spaces. This is indeed the case, as we show in \Cref{localg}. Let us introduce some constructions first.

	\begin{construction}\label{construction}
		Let $X$ be a dendroidal set and $S$ a subset of $X_{C_1}$. For any $f\colon a \to b$ in $S$, we define the $r_f \colon (\eta,\{b\})\longrightarrow (C_1,f)$ in $\dsets/X$ as the following commutative triangle
	
	\begin{center}
		\begin{tikzcd}
			\eta \arrow[rd, "b"'] \arrow[rr, "r", hook] &   & C_1 \arrow[ld, "f"] \\
			& X &                    
		\end{tikzcd}
	\end{center}
	where the arrow $r\colon\eta\to C_1$ is the inclusion of the edge $\eta$ into the root of $C_1$, or more familiary the map $\{1\}\colon [0]\to [1]$. 
	\noindent We write $S_{/X}$ for the set of morphisms of the form $r_f$, for $f\in S$.
	
	\end{construction} 
	
	\noindent Given a model category $\cM$ and some set $\cS$ of morphisms in it, we can talk about \emph{$\cS$-local} objects in $\cM$: there are the fibrant objects $M$ for which, for any morphism $s\colon A\to B$ in $\cS$, the morphism of mapping spaces $$ s_*\colon \mathsf{Map}_\cM(B,M)\longrightarrow \mathsf{Map}_\cM(A,M)$$ is a weak homotopy equivalence of spaces. For $\cM$ given by the covariant model structure for dendroidal left fibrations over a dendroidal set $X$, we have an explicit way of computing mapping spaces between fibrant-cofibrant objects: 
	\begin{rmk}\label{covmappingspace}
		If $(A,u)$ is a normal dendroidal set over $X$ and $(E,p)$ a dendroidal left fibration over $X$, we have an equivalence of spaces $$\mathsf{Map}_{\dsets/X}((A,u),(E,p))\simeq \hom_{X}(A,E)= \hom(A,E)\times_{\hom(A,X)} \{u\}.$$ 
	\end{rmk}
	
\noindent Maps in $S_{/X}$ have both cofibrant source and taget; we can rephrase $S_{/X}$-locality as follows.
	
	\begin{deff}\label{locdendrleft}
		A $(E,p)$ dendroidal left fibration over $X$ is $S_{/X}$-local if, for any $f$ in $S$, the morphism $$(r_f)^*\colon  \hom(C_1,E)\times_{\hom(C_1,X)} \{f\}\longrightarrow \hom(\eta,E)\times_{\hom(\eta,X)}\{b\}\simeq E_b= p^{-1}(b),$$  induced by the root inclusion $\eta \hookrightarrow C_1$
		is a weak homotopy equivalence of Kan complexes.\end{deff}

	\noindent We write $\cL_{S}\left(\dsets/X\right)$ for the left Bousfield localization of the covariant model structure on $\dsets/X$ at the set of arrows $S_{/X}$, which, recall, is left proper and cofibrantly generated. In this model structure, the fibrant objects are the $S_{/X}$-local dendroidal left fibrations, while the cofibrations are the normal monomorphisms over $X$. In particular, a morphism between $S_{/X}$-local left fibrations $\varphi\colon (E,p)\to (E',p')$ is a weak equivalence in the localization if and only if it is a covariant weak equivalence, hence a fibrewise weak homotopy equivalence. We call the model structure $$ \cL_S\left(\dsets/X\right)$$ the \emph{$S_{/X}$-local covariant model structure}. It is precisely this localization that encodes the compatibility of localization with the covariant model structure:
	
	\begin{prop}\label{localg} Let $X$ be a normal dendroidal set and $S\subseteq X_{C_1}$ a subset of morphisms. The localization map $\lambda\colon X\to X[S^{-1}]$ induces a Quillen adjunction
		$$ \adjunction{\lambda_!}{\cL_S\left(\dsets/X\right)}{\dsets/X[S^{-1}]}{\lambda^*}$$ between the $S_{/X}$-local covariant model structure on $\dsets/X$ and the covariant model structure on $\dsets/X[S^{-1}]$. Moreover, the adjunction is a Quillen equivalence.
	\end{prop}
	
	\begin{proof}
		We know that the adjunction $$  \adjunction{\lambda_!}{\dsets/X}{\dsets/X[S^{-1}]}{\lambda^*}$$ is a Quillen adjunction between covariant model structures (\cite[Proposition 9.62]{HeMo:SDHT}). To prove that it is still a Quillen adjunction after localizing on the left, it suffices to show that, given any dendroidal left fibration $(E,p)$ over $X$, the element $(E,p)$ is $S_{/X}$-local if and only if there exists a dendroidal left fibration $(Y,p')$ over $X[S^{-1}]$ and a covariant weak equivalence $(E,p)\xrightarrow{\sim}\lambda^*(Y,p)$ in $\dsets/X$. 
		
		\noindent Fix any such dendroidal left fibration $(E,p)$. Write $S\times C_1$, resp. $S\times J$, for the simplicial set given by the union $S\times C_1= \bigsqcup_{s\in S}C_1$, resp. $S\times J= \bigsqcup_{s\in S}J$, where $J$ is the nerve of the connected groupoid on two objects. 
		
		\noindent The pullback $$ E'\coloneqq (S\times {C_1})\times_X E \longrightarrow S \times C_1$$ is a left fibration of simplicial sets, and after \cite[Proposition 2.1.3.1]{Lu:HTT} it is a Kan fibration. As such, it can be factored as the composition $E'\to E''\to S\times C_1$, where $E'\to E''$ is a trivial Kan fibration and $E''\to S\times C_1$ is a minimal Kan fibration. As explained in \cite[\S 5.4]{GZ:CFHT}, there is an isomorphism $E''\simeq S\times C_1\times M$ for some minimal Kan complex $M$, and the map $E''\to S\times C_1$ is the projection. In particular, the map $ S\times J \times M\to S\times J$ is a minimal Kan fibration, fitting into the pullback diagram
		
		\[
		\begin{tikzcd}
			E'' \arrow[r, "\phi"] \arrow[d] & S\times J \times M \arrow[d]\\ S\times C_1\arrow[r] & S\times J
		\end{tikzcd}
		\]
		\noindent An argument due to Joyal (see \cite[Lemma 2.2.5]{KL:SMUF}) shows that we can find a trivial Kan fibration $Z\to S\times J \times M $ and an isomorphism $E'\simeq \phi^*(Z)$ over $E''$. Thus we obtain a commutative diagram 
		\begin{center}
			\begin{tikzcd}
				E \arrow[d] & {E'} \arrow[l] \arrow[d] \arrow[r] & Z \arrow[d]   \\
				X          & {S\times C_1} \arrow[l] \arrow[r]                       & S\times J
			\end{tikzcd}
		\end{center}
		where the right hand square is a pullback and the map $Z\to S\times  J$ is a Kan fibration. Define $$Y\coloneqq E\bigcup_{E'}Z \;$$ there is a canonical map $p'\colon Y \to X[S^{-1}]$, and its pullback along the surjection $(S\times J)\sqcup X \to X[S^{-1}]$ consists in the dendroidal left fibration $Z\sqcup E \to (S\times J)\sqcup X$. In particular, $p'$ is a dendroidal left fibration as well. Moreover, for any object $x$ of $X$, there is a weak homotopy equivalence of fibres $(E,p)_x\to \lambda^*(Y,p')_x$, hence a covariant weak equivalence $(E,p)\to \lambda^*(Y,p')$ over $X$, as wanted. This concludes the proof that there is an induced Quillen adjunction. It is straightforward to check that $\mathbb{R}\lambda^*$ is fully faithful, hence we conclude that the induced adjunction is a Quillen equivalence, as wanted.
	\end{proof}
	
	\begin{rmk} The above proof essentially extends to the dendroidal context Stevenson's proof of \cite[Proposition 5.11]{Ste:CMSSL} for simplicial sets.
	\end{rmk}
	
\subsection{Locally constant algebras}\label{sec:loccost}
\noindent We now perform an analogous localization of the projective model structure on the category $\alg_P(\ssets)$ of algebras over a (discrete) operad $P$. In this context, one must work in the framework of semi-model categories (see \cite{H:MMC}, \cite{Bar:LRMCLRBL}, \cite{WY:BLAOCO}; cf. \cite[\S 2]{C:WBLHIFMA} for a concise account).

\noindent Recall that a semi-model category consists of classes of weak equivalences, cofibrations, and fibrations satisfying the usual axioms, except that the lifting property for trivial fibrations and the factorization axioms are required only for maps with cofibrant source.

\noindent This weakening is necessary because left Bousfield localization of a model category at a set of maps yields a model structure under left properness, a condition that frequently fails for categories of operadic algebras. Nevertheless, such a localization still produces a semi-model structure, and it is in this sense that we proceed.

\noindent Let us give the following definition. This makes sense for any simplicial operad $P$.

\begin{deff}
	A simplicial $P$-algebra $A$ is \emph{$S$-locally constant} if $A$ sends $S$ to equivalences, that is for any $f\colon a \to b$ in $S$, the map $A(f)\colon A(a)\to A(b)$ is a weak homotopy equivalence of simplicial sets.
\end{deff}

\noindent Let us consider the operadic un/straightening equivalence of \cite[Theorem 5.9]{P:RDLF}, which consists, for any $\Sigma$-free operad $P$, in a Quillen equivalence \begin{equation}\label{qerect}
	\adjunction{\rho_!^P}{\dsets/\cN_d(P)}{\alg_P(\ssets)}{\rho^*_P}
\end{equation} between the covariant model structure on the left and the projective model structure on the right. The left adjoint $\rho_!^P$ is defined as follows.

\begin{deff}The functor $\rho_!^P\colon \dsets/\cN_d P \to \alg_P(\ssets)$ is the essentially unique cocontinuous functor characterized by the fact that, for any tree $T$ and morphism $\alpha\colon T \to \cN_d(P)$, the $P$-algebra $\rho_!^P(T,\alpha)$ is defined as $$ \mathsf{Ob}(P)\ni c \mapsto \rho_!^P(T,\alpha)(c)\simeq \mathsf{Env}(T)\underset{\mathsf{Env}(P)}{\times} \mathsf{Env}(P)_{/c},$$ where $\mathsf{Env}$ denotes the symmetric monoidal envelope \cite[\S2.2.4]{Lu:HA}.
\end{deff}

\noindent We can now prove the following

\begin{prop}\label{fibobjs}
		Let $P$ be a discrete $\Sigma$-free operad. The operadic un/straightening adjunction $(\rho_!^P,\rho^*_P)$ induces a Quillen equivalence
		$$ \adjunction{\rho_!^P}{\cL_S\left(\dsets/\cN_d(P)\right)}{\cL_S\left(\alg_P(\ssets)\right)}{\rho^*_P}$$ between the covariant model structure for $S$-local left fibrations and the left Bousfield localization of the projective model structure whose fibrant objects are $S$-locally constant $P$-algebras. 
\end{prop}

\begin{proof}
	
	\noindent In general, given a Quillen adjunction of model categories $\adjunction{F}{\cM}{\cN}{G}$ and a choice of (cofibrant) arrows $\cS$ in $\cM$, one can consider the left Bousfield localizations $\cL_\cS \cM$ on $\cM$ and the transfered localization on $\cN$, resp. $\cL_{\mathbb{L}F(\cS)}\cN$, of $\cM$ at $\cS$, resp. of $\cN$ at $\mathbb{L}F(\cS)$. It is a standard result of model categories (\cite{Hir:MCTL}) that, if $(F,G)$ is a Quillen equivalence, then there is an induced Quillen equivalence $\adjunction{F}{\cL_\cS\cM}{\cL_{\mathbb{L}F(\cS)}\cN}{G},$ where $\cL_{\mathbb{L}F(\cS)}\cN $ is the left Bousfield localization of $\cN$ at $\mathbb{L}F(\cS)$. This implies that we only need to characterize the fibrant objects of $ \cL_{\rho_!^P(S_{/\cN_d(P)})}\alg_P(\ssets)$ as the locally constant fibrant $P$-algebras.
	
	\noindent An object $A$ in $\cL_{\rho_!^P(\cS)}(\alg_P(\ssets))$ is fibrant if and only if it is projectively fibrant and $\rho_!^P(\cS_{/\cN_d P})$-local. 
	As the pair $(\rho^P_!,\rho_P^*)$ is a Quillen adjunction, this is equivalent to asking $\rho^*_P(A)$ to be $S_{/\cN_d P}$-local, that is, that $$ \mathsf{Map}_{\dsets/\cN_d(P)}((C_1,f),\rho^*_P(A))\longrightarrow \mathsf{Map}_{\dsets/\cN_d(P)}( (\eta,\{b\}),\rho^*_P(A))$$ is a weak homotopy equivalence. 
	As $A$ is projectively fibrant, $\rho^*_P(A)$ is a dendroidal left fibration, so in particular local with respect to the map ${(\eta,\{a\})\to (C_1,f)}$ induced by the leaf inclusion $\ell \colon \eta\hookrightarrow C_1$, so the above is an equivalence if and only if the map of spaces $$\mathsf{Map}_{\dsets/\cN_d(P)}((\eta,\{a\}),\rho^*_P(A))\longrightarrow  \mathsf{Map}_{\dsets/\cN_d(P)} ((\eta, \{b\}),\rho^*_P(A))$$ is a weak homotopy equivalence. After \Cref{covmappingspace}, there is an equivalence of spaces $$  \mathsf{Map}_{\dsets/\cN_d(P)}((\eta,\{x\} )\simeq{ \rho^*_P(A)}_x.$$
	\noindent As the right adjoint $\rho^*_P$ enjoys the property that for any $P$-algebra $A$ and any object $x$ of $P$, there is an equivalence $$\rho_P^*(A)_c\simeq A(c) $$ between the fibre of $\rho_P^*(A)$ over $c$ and the value of $A$ on $c$ (\cite[Lemma 2.9]{P:RDLF}), we conclude that $\rho_P^*(A)$ is $S_{/\cN_dP}$-local if and only if the map $$A(f)\colon A(a)\longrightarrow A(b)$$ is a weak homotopy equivalence of spaces, which means that $A$ is $S$-locally constant. This concludes the proof.
\end{proof} 

\begin{rmk}
	The above result only depends on the fact that $\rho_!^P$ preserves and detects fibrant $S$-locally constant algebras, which follows from the equivalence $\rho_!^P(F)_a\simeq F(a)$ between the
	fiber over an object $c$ of $P$ of the unstraightening of an algebra $F$ and evaluation of $F$ at $c$.
	\end{rmk}
\noindent As a corollary, by instantiating Proposition \ref{fibobjs} with $P=(\omg/X)^\otimes$, which is  $\Sigma$-free (Remark \ref{rmk:sfree}), we obtain a new version of the un/straightening equivalence for a normal dendroidal $\infty$-operad $X$, that is an equivalence of semi-model categories
$$\adjunction{\rho_!^{(\omg/X)^\otimes}}{ \cL_{\cR_X}\left(\dsets/\cN_d(\omg/X)^\otimes\right)}{\cL_{\cR_X}\left(\alg_{(\omg/X)^\otimes}(\ssets)\right)}{\rho^*_{(\omg/X)^\otimes}}$$ between the covariant model structure for $\cR_X$-local dendroidal left fibrations over $\cN_d \omg/X$ and the projective semi-model structure for $\cR_X$-locally constant $(\omg/X)^\otimes$-algebras.

\noindent We may interpret this as a 'local' un/straightening equivalence. As a corollary, we deduce the following local description for algebras.
\begin{coro}\label{corofip}

	Let $X$ be a normal dendroidal $\infty$-operad. There is a zig-zag of Quillen equivalences $$ \alg_{W_!(X)}(\ssets)\xleftarrow{\sim}\bullet\xrightarrow{\sim}\cL_{\cR_X}\left(\alg_{(\omg/X)^\otimes}(\ssets)\right)$$ between the projective model structures for simplicial $W_!(X)$-algebras and for $\cR_X$-locally constant algebras on $(\omg/X)^\otimes$.
	\end{coro}
\begin{proof}
By \cite[Proposition 2.4]{He:AOIO}, if $f\colon X \to Y$ is an operadic weak equivalence between normal dendroidal sets, the induced adjunction $(f_!,f^*)$ on over categories is a Quillen equivalence between covariant model structures. As this is natural in the dendroidal sets, applying \Cref{fibobjs} we obtain that the root functor yields a Quillen equivalence $$\adjunction{(\scriptr_X)_!}{\cL_{\cR_X}\left(\dsets/\cN_d(\omg/X)^\otimes\right) }{\dsets/X}{\scriptr_X^*}. $$ Combining this with Proposition \ref{fibobjs} and the classical \cite[Theorem 2.7]{He:AOIO}, we get the wanted zig-zag of Quillen equivalences of semi-model categories.
\end{proof}

\addtocontents{toc}{\SkipTocEntry}

\providecommand{\bysame}{\leavevmode\hbox to3em{\hrulefill}\thinspace}
\providecommand{\MR}{\relax\ifhmode\unskip\space\fi M`R }
\providecommand{\MRhref}[2]{%
	\href{http://www.ams.org/mathscinet-getitem?mr=#1}{#2}}
\providecommand{\href}[2]{#2}

\bibliographystyle{alpha}
\bibliography{Root}

\end{document}